\documentclass[11pt]{article}
\usepackage{style}
\author{Julia Gaudio\footnote{Northwestern University.
    Email: julia.gaudio@northwestern.edu. Part of this work was completed while J.G. was at the Massachusetts Institute of Technology.} and Elchanan Mossel\footnote{Massachusetts Institute of Technology.
    Email: elmos@mit.edu. E.M is partially supported by Vannevar Bush Faculty Fellowship ONR-N00014-20-1-2826, NSF awards CCF 1918421 and DMS-1737944, ARO MURI grant W911NF1910217 and by a Simons Investigator award.}}

\title{Shotgun Assembly of Erd\H{o}s-R\'enyi Random Graphs}

\begin{document}
\maketitle

\begin{abstract}
Graph shotgun assembly refers to the problem of reconstructing a graph from a collection of local neighborhoods. In this paper, we consider shotgun assembly of \ER random graphs $G(n, p_n)$, where $p_n = n^{-\alpha}$ for $0 < \alpha < 1$. We consider both reconstruction up to isomorphism as well as exact reconstruction (recovering the vertex labels as well as the structure). We show that given the collection of distance-$1$ neighborhoods, $G$ is exactly reconstructable for $0 < \alpha < \frac{1}{3}$, but not reconstructable for $\frac{1}{2} < \alpha < 1$. Given the collection of distance-$2$ neighborhoods, $G$ is exactly reconstructable for $\alpha \in \left(0, \frac{1}{2}\right) \cup \left(\frac{1}{2}, \frac{3}{5}\right)$, but not reconstructable for $\frac{3}{4} < \alpha < 1$.
\end{abstract}

\section{Introduction}
Graph shotgun assembly refers to the reconstruction of a graph from a collection of local neighborhoods. The terminology is inspired by DNA shotgun assembly, which is the problem of determining a DNA sequence from multiple short sequences. Mossel and Ross \cite{Mossel2019} introduced the problem of graph shotgun assembly for generative models. 

We now describe the graph shotgun assembly problem. Let $G = (V,E)$ be a graph with vertices labeled from $[n] = \{1, 2, \dots, n\}$. For a vertex $v \in V$, let $\mathcal{N}_r(v)$ be the subgraph induced by all vertices at distance at most $r$ from $v$ in the graph $G$. 
For each vertex $v \in V$, we observe the neighborhood $\mathcal{N}_r(v)$. The central vertex $v$ in this neighborhood is labeled, and other vertices are unlabeled.

We say two graphs $G = (V_G, E_G)$ and $H = (V_H, E_H)$ are \emph{isomorphic} if there exists a bijection $f: V_G \to V_H$ such that $(u,v) \in E_G \iff (f(u), f(v)) \in E_H$. In that case, we write $G \sim H$. The goal is to reconstruct $G$ from the set of $r$-neighborhoods, up to isomorphism. In other words, we wish to find a graph $H$ isomorphic to $G$ given the unlabeled $r$-neighborhoods of $G$.

We say that two graphs $H_1(V, E_1)$ and $H_2(V, E_2)$ have the same $r$-neighborhoods if for all $v \in V$, the $r$-neighborhood around $v$ in $H_1$ is isomorphic to the $r$-neighborhood around $v$ in $H_2$. In that case we write $H_1 \sim_r H_2$. We say that the graph $G$ is reconstructable from its $r$-neighborhoods if for all graphs $H = (V, E_H)$, $G \sim_r H \implies G \sim H$. In other words, $G$ is reconstructable from its $r$-neighborhoods if the only other graphs with the same $r$-neigborhoods are those that are isomorphic to $G$.

In addition to the above reconstructability question, we consider exact reconstructability. We say that a graph $G$ is \emph{exactly reconstructable} if it is possible to recover $G$ (rather than just $H \sim G$). Note that exact reconstructability implies reconstructability up to isomorphism. For an example where reconstruction up to isomorphism is possible, but exact reconstruction is not, consider the graph on vertices $\{1,2,3,4\}$ with edges $(1,2)$ and $(3,4)$. This graph has the same neighborhoods as the graph on vertices $\{1,2,3,4\}$ with edges $(1,3)$ and $(2,4)$

Along with introducing the problem of graph shotgun assembly for generative models, \cite{Mossel2019} considered several random structure models. First, they considered the $d$-dimensional $n$-lattice, where the vertices are additionally labeled i.i.d. according to some distribution on $[q]$. Next, they considered the reconstruction of \ER random graphs, both in the dense and sparse case. In the sparse case where $p_n = \frac{\lambda}{n}$, it was shown that the asymptotic threshold for reconstructability is a radius of $r = \log(n)$, as long as $\lambda \neq 1$. In the dense case, it was shown that $r = 3$ suffices for reconstruction when $\lim_{n \to \infty} \frac{n p_n}{\log^2(n)} = \infty$. Finally, \cite{Mossel2019} introduced the random jigsaw puzzle, for which there has been much follow-up work (\cite{Balister2017}, \cite{Bordenave2020}, \cite{Martinsson2016}, \cite{Martinsson2019}, \cite{Nenadov2017}). A jigsaw puzzle is an $n \times n$ grid where the border of two adjacent pieces is uniformly assigned to one of $q$ possible shapes called ``jigs.'' Martinsson (\cite{Martinsson2016}, \cite{Martinsson2019}) showed that reconstruction is possible if $q \geq (2+\epsilon)n$, and impossible if $q \leq \frac{2}{\sqrt{e}} n$. Other follow-up work includes a study of reconstruction in random regular graphs \cite{Mossel2015} and the hypercube \cite{Przykucki2019}.

In this paper, we continue the study of graph shotgun assembly for \ER random graphs. Let $G = (V, E)$ be a labeled Erd\H{o}s-R\'enyi graph drawn from the model $G(n,p_n)$.  The following general observation was used in \cite{Mossel2019} to establish reconstruction from $3$-neighborhoods.
\begin{lemma}[Lemma 2.4 in \cite{Mossel2019}]\label{lemma:overlap}
If $\mathcal{N}_{r-1}(v) \not \sim \mathcal{N}_{r-1}(w)$ for all vertices $v \neq w$, then there is an algorithm for recovering the graph from $r$-neighborhoods.
\end{lemma}
The algorithm is as follows: Given a neighborhood $\mathcal{N}_r(v)$, we can uniquely label any neighbor $v_0$ of $v$ by examining $\mathcal{N}_{r-1}(v_0)$, which is contained within $\mathcal{N}_r(v)$. We call this algorithm the ``overlap method." Using Lemma \ref{lemma:overlap}, \cite{Mossel2019}  showed the following.
\begin{theorem}[Theorem 4.5 in \cite{Mossel2019}]\label{theorem:3-neighbourhoods-original}
$G(n,p_n)$ is reconstructable from its $3$-neighborhoods with high probability, for $p_n$ satisfying $\lim_{n \to \infty} \frac{n p_n}{\log^2(n)} = \infty$. 
\end{theorem}
The proof of Theorem \ref{theorem:3-neighbourhoods-original} shows that no two
vertices $u, v \in V$ have the same degree neighborhoods with high
probability. (We say that two vertices $u, v$ have the same degree neighborhoods if they have the same degree and the degrees of their neighbors are equal as multi-sets.)

In this paper, we extend the results of \cite{Mossel2019}, by establishing regimes where reconstruction is possible and when it is not, given the collection of $1$- or $2$-neighborhoods.

\subsection{Main Results}\label{sec:main-results}
We first investigate reconstructablity from $1$-neighborhoods (Section \ref{sec:1-neighbourhoods}). First, we establish a range where reconstruction is possible using a ``fingerprinting'' idea from \cite{Yartseva2016}, which studied reconstruction in another random graph model.
\begin{theorem}\label{theorem:1-reconstructable}
Let $p_n = n^{-\alpha}$ for $0 < \alpha < \frac{1}{3}$. Then $G(n, p_n)$ is exactly reconstructable from its $1$-neighborhoods with high probability. 
\end{theorem}
We outline the proof. For vertices $u, v$ such that $(u,v) \in E$, let $H_{u,v}$ denote the graph induced by the common neighbors of $u$ and $v$. 
The graph $H_{u,v}$ is a ``fingerprint'' for the edge $(u,v)$ (note that $H_{u,v} \sim H_{v,u}$). The following lemma establishes that under a uniqueness of fingerprints condition, the graph can be reconstructed.
\begin{lemma}[Fingerprint Lemma]\label{lemma:fingerprint}
Suppose $H_{u,v} \sim H_{x,y}$ if and only if $(u,v)$ and $(x,y)$ are the same edge. Then we can exactly reconstruct $G$ from the collection of $1$-neighborhoods.
\end{lemma}
Given the Fingerprint Lemma, it remains to establish the uniqueness of the $H_{u,v}$ graphs. 

We then show a negative result about the reconstructability of $G(n, p_n)$ from $1$-neighborhoods using a counting argument.
\begin{theorem}\label{theorem:1-non-reconstructable-part-1}
Let $p_n = n^{-\alpha}$ for $\frac{1}{2} < \alpha < 1$. Then with high probability, $G(n, p_n)$ cannot be reconstructed from its $1$-neighborhoods.
\end{theorem}

Next, we consider reconstruction from $2$-neighborhoods (Section \ref{sec:2-neighbourhoods}). Theorem \ref{theorem:2-reconstructable} establishes that $G(n,p_n)$ is exactly reconstructable from $2$-neighborhoods for $p_n = n^{-\alpha}$ where $0 < \alpha < \frac{3}{5}$.  When $0 < \alpha < \frac{1}{2}$, the graph has diameter $2$ with high probability (Lemma \ref{lemma:diameter-2}). We then construct a canonical labeling, which is guaranteed by Lemma \ref{lemma:canonical-labelling}. For values $\frac{1}{2} < \alpha < \frac{3}{5}$, we demonstrate reconstructability by a fingerprint argument similar to the proof of Theorem \ref{theorem:1-reconstructable}.

\begin{theorem}\label{theorem:2-reconstructable}
Let $p_n = n^{-\alpha}$ for $\alpha \in \left(0, \frac{1}{2}\right) \cup \left(\frac{1}{2}, \frac{3}{5}\right)$. Then $G(n, p_n)$ is exactly reconstructable from its $2$-neighborhoods with high probability. Moreover, there is an efficient algorithm for reconstruction when $\alpha < \frac{1}{2}$.
\end{theorem}

On the other hand, when $\frac{3}{4} < \alpha < 1$, we show by a counting argument that $G(n, p_n)$ cannot be reconstructed from its $2$-neighborhoods. 
\begin{theorem}\label{theorem:2-non-reconstructable}
Let $p_n = n^{-\alpha}$ for $\frac{3}{4} < \alpha < 1$. Then with high probability, $G(n, p_n)$ cannot be reconstructed from its $2$-neighborhoods.
\end{theorem}

These above results do not cover $\alpha \in \left\{\frac{1}{2}\right\} \cup \left [\frac{3}{5}, \frac{3}{4} \right]$. Proposition \ref{proposition:overlap} gives a partial answer for the range $\alpha \in \left( \frac{2}{3}, \frac{3}{4} \right]$.
\begin{proposition}\label{proposition:overlap}
Let $p_n = n^{-\alpha}$ for $\frac{2}{3} < \alpha < 1$. Then with high probability, $G(n, p_n)$ cannot be exactly reconstructed from its $2$-neighborhoods using the overlap method.
\end{proposition}

Finally, recall that the center of each neighborhood was assumed to be labeled. In Section \ref{section:find-center}, we show how to find the center if it is unlabeled.

\subsection{Follow-up work}
After posting of this paper to arXiv, Huang and Tikhomirov \cite{Huang2021} contributed further results on the reconstruction of \ER random graphs from $1$-neighborhoods. Their first result refined Theorem \ref{theorem:1-non-reconstructable-part-1}, showing that $G(n, p_n)$ is a.a.s. non-reconstructable from its $1$-neighborhoods if $p_n$ satisfies $p_n = \omega(n^{-1} \log n)$ and $p_n = o(n^{-\frac{1}{2}})$. Their second result showed that there exist universal constants $C, c > 0$ such that if $n^{-\frac{1}{2}} \log^C(n) \leq p_n \leq \frac{c}{n}$ for large $n$, then $G(n, p_n)$ is a.a.s. reconstructable from its $1$-neighborhoods. Therefore, $p = \tilde{\Theta}(n^{-\frac{1}{2}})$ marks the transition between reconstructability and non-reconstructability, addressing an open problem that we had raised.

\subsection{Preliminaries}
We collect some results that will be used to establish the theorems.
\begin{lemma}[Chernoff Bound \cite{Mitzenmacher2017}]\label{lemma:chernoff}
Let $\{X_i\}_{i=1}^m$ be independent indicator random variables, and let $X = \sum_{i=1}^m X_i$. Then for any $\epsilon > 0$,
\[\mathbb{P}\left(X \leq (1-\epsilon) \mathbb{E}[X]\right) \leq \exp \left(-\frac{\epsilon^2}{2} \mathbb{E}[X]\right) \text{~~and~~} \mathbb{P}\left(X \geq (1+\epsilon) \mathbb{E}[X]\right) \leq \exp \left(-\frac{\epsilon^2}{2 + \epsilon} \mathbb{E}[X]\right).\]
\end{lemma}
\begin{corollary}\label{corollary:chernoff}
Let $X$ be a binomial random variable with parameters $(n, n^{-\beta})$ for constant $\beta > 0$. For any constant $c > 0$,
\[\mathbb{P}\left(X \geq n^{1-\beta + c} \right) \leq \exp \left( - \frac{\left(n^c -1\right)^2}{n^c + 1} n^{1-\beta}\right)  =\exp \left( - \Theta(n^{1-\beta + c})\right). \]
\end{corollary}

\begin{lemma}\label{lemma:binomial-domination}
Let $X$ and $Y$ be random variables such that conditioned on $Y$, $X$ is distributed as a binomial random variable with parameters $(Y, p)$. Let $Z(m)$ be a binomial random variable with parameters $(m, p)$. Then
\[\mathbb{P}\left(X \leq t_1 ~\Big|~ Y \geq t_2 \right) \leq \mathbb{P}\left(Z(t_2) \leq t_1\right) \text{~~and~~} \mathbb{P}\left(X \geq t_1 ~\Big|~ Y \leq t_2 \right) \leq \mathbb{P}\left(Z(t_2) \geq t_1\right).\]
\end{lemma}
\begin{proof}
The inequalities follow from the observation that $\mathbb{P}\left(Z(m_1) \geq t \right) \leq\mathbb{P}\left( Z(m_2) \geq t\right)$
for $m_1 < m_2$.
\end{proof}

\section{Reconstruction from $1$-neighborhoods}\label{sec:1-neighbourhoods}
\begin{proof}[Proof of Lemma \ref{lemma:fingerprint}]
To determine whether a pair of vertices $u, v$ is connected by an edge, we examine the neighborhoods of $u$ and $v$, observing graphs $H_{u,u_0}$ and $H_{v,v_0}$ for neighbors $u_0 \sim u$ and $v_0 \sim v$. We declare that $u$ and $v$ are connected if we observe vertices $u_0 \sim u$ and $v_0 \sim v$ in the neighborhoods of $u$ and $v$ respectively such that $H_{u,u_0} \sim H_{v,v_0}$. Clearly this occurs if $(u,v) \in E$. For the other direction, suppose there exist vertices $u_0 \sim u$ and $v_0 \sim v$ such that $H_{u,u_0} \sim H_{v,v_0}$. By the assumption of the lemma, $(u,u_0)$ and $(v,v_0)$ are the same edge, and we conclude that $u$ and $v$ are connected. Continuing this process, we recover the whole graph.
\end{proof}

\begin{proof}[Proof of Theorem \ref{theorem:1-reconstructable}]
Applying Lemma \ref{lemma:fingerprint}, it suffices to show that with high probability, whenever $(u,v), (x,y) \in E$ are distinct edges, then $H_{u,v} \nsim H_{x,y}$.

Let $W_{ab}$ denote the number of shared neighbors of vertices $a$ and $b$. If $H_{u,v} \sim H_{x,y}$, then $W_{uv} = W_{xy}$. Let $Z = \mathbbm{1}\{u \sim x, v \sim x\} +\mathbbm{1}\{u \sim y, v \sim y\}$, and suppose $W_{uv} = W_{xy}  = \lambda + Z$. In other words, $\lambda$ is the number of common neighbors of $u$ and $v$ excluding $x$ and $y$. Let $Y$ be the number of shared neighbors of $u$, $v$, $x$, and $y$. Suppose also that $Y = \mu$. Let $G_1$ be the subgraph induced by the $\lambda - \mu$ shared neighbors of $u$ and $v$ (excluding $x$ and $y$) that are not neighbors of both $x$ and $y$. Let $G_2$ be the subgraph induced by all $\lambda + Z$ shared neighbors of $x$ and $y$. The graphs $G_1$ and $G_2$ are disjoint. Observe that if $H_{u,v} \sim H_{x,y}$, then $G_1$ must be a subgraph of $G_2$, which we write as $G_1 \subset G_2$. 

We upper-bound the probability of the event $\{G_1 \subset G_2\}$ by the First Moment Method. Observe that $G_1$ is an \ER graph, while we may have already revealed some edges of $G_2$. We note that up to $2\mu + 1$ edges have already been revealed in $G_2$; if $u$, $v$, $x$, and $y$ form a clique, then the edges from $u$ and $v$ to the shared neighbors of $u$, $v$, $x$, and $y$ have been revealed, contributing $2\mu$ edges, with the last edge revealed being $(u,v)$ itself. Let $E(H)$ denote the number of edges of a graph $H$. We have
\begin{align}
&\mathbb{P}\left(H_{u,v} \sim  H_{x,y} ~\big|~ W_{uv} = W_{xy} = \lambda + Z, Y = \mu, E(G_1) = k \right)  \nonumber\\
&\leq \mathbb{P}\left( G_1 \subset G_2 ~\big|~ W_{uv} = W_{xy} = \lambda + Z, Y = \mu, E(G_1) = k \right) \nonumber \\
&\leq \binom{\lambda + 2}{\lambda - \mu} (\lambda - \mu)! p_n^{k-2\mu-1} \nonumber \\
&\leq (\lambda + 2)^{\lambda - \mu} p_n^{k-2\mu-1} \nonumber\\
&\leq \exp \left(\lambda \log(n) + (k-2\mu-1) \log(p_n) \right) \nonumber\\
&= \exp \left(\log(n) \left(\lambda - \alpha(k - 2\mu-1) \right)\right)\nonumber\\
&\leq \exp \left(\log(n) \left((2\alpha+1) \lambda + \alpha - \alpha k \right)\right) \label{eq:bound-fingerprint}
\end{align}
We now provide high probability bounds on $W_{uv} - Z \leq W_{uv}$ and $E(G_1)$. Observe that $W_{uv}$ is distributed as a binomial random variable with parameters $(n-2, p_n^2)$. Let $c > 0$, to be determined. By Corollary \ref{corollary:chernoff},
\begin{align*}
\mathbb{P}\left(W_{uv} \geq n^c (n-2)p_n^2\right) &\leq  \exp\left(-\Theta\left(n^{1+c - 2\alpha} \right) \right).
\end{align*}
For this to be a high probability bound, we need $1+c-2\alpha > 0 \iff c > 2\alpha -1$. In that case, we obtain
\begin{align}
\mathbb{P}\left(W_{uv} \geq n^{1+c-2\alpha} \right) &= o \left( \frac{1}{n^4}\right). \label{eq:Y-bound}
\end{align}
Next, observe that $E(G_1)$ is a binomial random variable with parameters $\left( \binom{\lambda - \mu}{2}, p_n\right)$, conditioned on $W_{uv} - Y -Z = \lambda - \mu$. In turn, $W_{uv} - Y -Z$ is a binomial random variable. If the vertices $u$, $v$, $x$, and $y$ are distinct, then $W_{uv} - Y -Z$ is binomial with parameters $\left(n-4, p_n^2(1-p_n^2)\right)$. Otherwise, $W_{uv} - Y -Z$ is binomial random variable with parameters $\left(n-3, p_n^2(1-p_n)\right)$. Fix $0 < \epsilon < \frac{1}{2}$. In the case of distinct vertices, we have
\begin{align*}
\mathbb{P}\left(W_{uv} - Y -Z \leq (1-\epsilon) (n-4)   p_n^2(1-p_n^2) \right) &\leq \exp \left(- \frac{\epsilon^2}{2} (n-4)   p_n^2(1-p_n^2)\right)\\
& = \exp \left(- \Theta\left(n^{1 -2\alpha} \right) \right) = o \left(\frac{1}{n^4}\right).
\end{align*}
A similar conclusion holds if the vertices are not distinct, so that
\begin{align}
\mathbb{P}\left(W_{uv} - Y -Z \leq \frac{1}{2} n p_n^2 \right) &= o \left( \frac{1}{n^4}\right),\label{eq:vertex-bound}
\end{align}
for $n$ sufficiently large. Applying \eqref{eq:vertex-bound} and Lemma \ref{lemma:binomial-domination}, we obtain
\begin{align}
\mathbb{P} \left( E(G_1) \leq (1 - \epsilon ) p_n \binom{\frac{1}{2}n p_n^2 }{2} \right) &\leq  \exp\left(- \frac{\epsilon^2}{2}p_n \binom{\frac{1}{2}np_n^2 }{2}   \right) + o \left( \frac{1}{n^4}\right) \nonumber \\
&= \exp\left(- \Theta\left(n^{2  -5\alpha} \right) \right) + o \left( \frac{1}{n^4}\right). \label{eq:E-bound}
\end{align}
Using \eqref{eq:Y-bound} and \eqref{eq:E-bound} in \eqref{eq:bound-fingerprint}, we arrive at
\begin{align*}
\mathbb{P}\left(H_{u,v} \sim H_{x,y}\right) &\leq \exp \left(\Theta\left(n^{1+c-2\alpha}\log(n) \right) - \Theta\left(n^{2 -5\alpha} \log(n) \right) \right) + o \left(\frac{1}{n^4}\right).
\end{align*}
For this to be a high probability bound, we need $1+c -2\alpha < 2  -5\alpha \iff c  < 1-3\alpha $.

Summarizing, the conditions are $\max\{0, 2\alpha-1\} < c < 1- 3\alpha$. This is a consistent condition, since $\alpha < \frac{1}{3}$. We choose $c = \frac{1}{2} \left(\max\{0, 2\alpha -1\} +1- 3\alpha \right)$. Applying a union bound, we conclude that $H_{u,v} \nsim H_{x,y}$ whenever $(u,v)$ and $(x,y)$ are distinct edges, with high probability.
\end{proof}
\begin{remark}
Intuitively, $H_{u,v}$ looks like $G\left(n p_n^2, p_n\right)$, and therefore we expect that if it is supercritical, i.e. $n p_n^3 \gg 1$, then it will not be isomorphic to an independent copy. Of course, $H_{u,v}$ and $H_{x,y}$ are not independent, which needs to be accounted for in the proof. Similarly, for radius $r$, the expected number of vertices in the analogue of $H_{u,v}$ is of order $n \left(\nicefrac{(np_n)^r}{n} \right)^2 = n^{2r-1} p_n^{2r}$, so reconstruction should be possible when $G\left(n^{2r-1} p_n^{2r}, p_n \right)$ is supercritical, i.e. $p_n \gg n^{\frac{1-2r}{1+2r}}$. On the other hand, \cite{Mossel2019} showed that for $r=3$ it suffices that $\lim_{n \to \infty} \frac{np_n}{\log^2(n)} = \infty$, a weaker condition than $\lim_{n \to \infty} n^{\frac{5}{7}}p_n = \infty$. This suggests that the fingerprint argument may not be optimal.\end{remark}

\begin{proof}[Proof of Theorem \ref{theorem:1-non-reconstructable-part-1}]
Let $\mathcal{N}_r(G)$ denote the collection of $r$-neighborhoods of a graph $G$. Consider a particular reconstruction algorithm. Given a collection of neighborhoods, the algorithm outputs a graph. We will find a set $S$ of neighborhood collections such that with high probability, $\mathcal{N}_1(G) \in S$. Next, suppose we know that $|E| = m$. Given this information, there are $\binom{\binom{n}{2}}{m}$ possible labeled graphs, which are sampled with equal probability. The algorithm maps each element of $S$ to an isomorphism class, which corresponds to at most $n!$ graphs. Therefore, conditioned on $|E| = m$, the algorithm fails with probability at least 
\[\mathbb{P}\left(\mathcal{N}_1(G) \in S ~\big|~ |E| = m \right) - \frac{n!|S|}{\binom{\binom{n}{2}}{m}} = \frac{\binom{\binom{n}{2}}{m} - n!|S|}{\binom{\binom{n}{2}}{m}} - \mathbb{P}(\mathcal{N}_1(G) \not \in S ~\big|~ |E| = m).\]
To see this, observe that the algorithm fails whenever $G$ is such that $\mathcal{N}_1(G) \in S$ but $G$ is not one of the at most $n! |S|$ graphs corresponding to the algorithm's output on $|S|$. Let $S_E \subset \left[ \binom{n}{2} \right]$. The overall failure probability of the algorithm is then at least
\begin{align*}
&\sum_{m \in S_E} \frac{\binom{\binom{n}{2}}{m} - n!|S|}{\binom{\binom{n}{2}}{m}}  \mathbb{P}\left( |E| = m\right) - \mathbb{P}\left(\mathcal{N}_1(G) \not \in S \cap |E| \in S_E\right) \\
&\geq \mathbb{P}\left(|E| \in S_E\right) \min_{m \in S_E} \frac{\binom{\binom{n}{2}}{m} - n!|S|}{\binom{\binom{n}{2}}{m}} - \mathbb{P}\left(\mathcal{N}_1(G) \not \in S \right).
\end{align*}
Therefore, it suffices to show $\mathcal{N}_1(G) \in S$ with high probability, $|E| \in S_E$ with high probability, and $n!|S| = o\left(\binom{\binom{n}{2}}{m}\right)$ for all $m \in S_E$.

We now construct the set $S$. Fix $0 < \epsilon <1$, and let $q_n = (1+\epsilon) p_n$. Let $c> 0$, to be determined, and let $t_n = (1 + n^c) p_n$. Let $S$ be the set of possible $1$-neighborhood collections where the degree of each central vertex is less than $q_n n$ and each neighborhood has fewer than  $\frac{1}{2} q_n^2 t_n n^2$ neighbor edges. We need to show that the collection of $1$-neighborhoods of $G$ appears in $S$ with high probability. By a Chernoff bound,
\begin{align}
\mathbb{P}\left(\deg(v) <  q_n n \right) &\geq \mathbb{P}\left(\deg(v) <  q_n (n-1)\right) \nonumber \\
&= 1-  \mathbb{P}\left(\deg(v) \geq  (1+\epsilon) p_n (n-1)\right) \nonumber \\
&\geq 1 -  \exp \left(-\frac{\epsilon^2 p_n (n-1) }{3} \right). \label{eq:single-degree-bound}
\end{align}
Let $Z(v)$ be the number of edges among the neighbors of $v$. Conditioned on $\deg(v)$, $Z(v)$ is distributed as a binomial random variable with parameters $\left( \binom{\deg(v)}{2}, p_n \right)$. We have
\begin{align}
\mathbb{P}\left( Z(v) \geq \frac{1}{2} q_n^2 t_n n^2 \right)  &\leq  \mathbb{P}\left( Z(v) \geq \frac{1}{2} q_n^2 t_n n^2   ~\Big|~ \deg(v) <  q_n n \right) + \mathbb{P}\left(\deg(v) \geq  q_n n \right). \label{eq:Z-bound}
\end{align}
By Lemma \ref{lemma:binomial-domination} and a Chernoff bound,
\begin{align}
\mathbb{P}\left( Z(v) \geq \frac{1}{2} q_n^2 t_n n^2   ~\Big|~ \deg(v) <  q_n n \right)  &\leq \mathbb{P}\left( Z(v) \geq \frac{1}{2} q_n^2 t_n n^2   ~\Big|~ \deg(v) =  \lfloor q_n n \rfloor \right) \nonumber \\
&\leq \mathbb{P}\left( Z(v) \geq \binom{\lfloor q_n n \rfloor}{2} t_n   ~\Big|~ \deg(v) =  \lfloor q_n n \rfloor \right) \nonumber \\
&\leq \exp\left(- \frac{n^{2c}}{2 + n^c}\binom{\lfloor q_n n \rfloor}{2} p_n  \right). \label{eq:Z-bound-conditional}
\end{align}
Substituting \eqref{eq:single-degree-bound} and \eqref{eq:Z-bound-conditional} into \eqref{eq:Z-bound}, we arrive at
\begin{align*}
\mathbb{P}\left( Z(v) \geq \frac{1}{2} q_n^2 t_n n^2 \right) &\leq  \exp\left(- \frac{n^{2c}}{2 + n^c}\binom{\lfloor q_n n \rfloor}{2} p_n  \right)  +  \exp \left(-\frac{\epsilon^2 p_n (n-1) }{3} \right).
\end{align*}
By a union bound, 
\begin{align*}
&\mathbb{P}\left( \bigcap_{v \in V} \left \{\deg(v) <  q_n n \cap Z(v) < \frac{1}{2}q_n^2 t_n n^2 \right \}\right) \\
&\geq 1  - n \exp\left(- \frac{n^{2c}}{2 + n^c}\binom{\lfloor q_n n \rfloor}{2} p_n  \right) - n \exp \left(-\frac{\epsilon^2 p_n (n-1) }{3} \right)\\
&= 1  - n \exp\left(-\Theta (n^{2 + c-3\alpha})\right) - n \exp\left(-\Theta(n^{1-\alpha})\right).
\end{align*}
We therefore need $\alpha < 1$ and $2 + c - 3\alpha > 0 \iff c > 3\alpha -2$ for this to be a high probability bound. If these constraints are satisfied, then $\mathcal{N}_1(G) \in S$ with high probability.

We now bound $|S|$. For clarity of presentation, we omit floor functions; the resulting off-by-one errors will not affect asymptotics. In each neighborhood, there are $q_n n$ choices for the degree of the central vertex, and $\frac{1}{2} q_n^2 t_n n^2$ choices for the number of neighbor edges. If $t_n = o(1) \iff c < \alpha$, then the number of choices for the set of edges is upper-bounded by  $\binom{\binom{q_n n}{2}}{\frac{1}{2} q_n^2 t_n n^2}$. Therefore,
\begin{align}
|S| &\leq \left(q_n n \cdot \frac{1}{2} q_n^2 t_n n^2 \binom{\binom{q_n n}{2}}{\frac{1}{2} q_n^2 t_n n^2}\right)^n \leq \left(\frac{1}{2} q_n^3 t_n n^3 \left(\frac{e(q_n n)^2}{q_n^2 t_n n^2}\right)^{\frac{1}{2} q_n^2 t_n n^2} \right)^n \nonumber \\
&= \left(\frac{1}{2} q_n^3 t_n n^3 \left(\frac{e}{t_n}\right)^{\frac{1}{2} q_n^2 t_n n^2} \right)^n \label{eq:S-size}.
\end{align}

We now choose $S_E$ to be $\{m \in \mathbb{N} : \left|m - \binom{n}{2} p_n \right| < \epsilon \binom{n}{2}\}$. By a Chernoff bound,
\begin{align*}
\mathbb{P}\left(|E| \not \in S_E \right)) = \mathbb{P}\left( \left| |E| - \binom{n}{2} p_n \right| \geq \epsilon \binom{n}{2} \right) &\leq 2 \exp \left(-\frac{\epsilon^2}{3} \binom{n}{2} p_n\right) = o(1).
\end{align*}
Therefore, $|E| \in S_E$ with high probability.

We now compute
\[\min_{m \in S_E} \frac{\binom{\binom{n}{2}}{m} - n!|S|}{\binom{\binom{n}{2}}{m}} = 1 - \max_{m \in S_E} \frac{n!|S|}{\binom{\binom{n}{2}}{m}} .\]
We have
\begin{align}
\min_{m \in S_E} \binom{\binom{n}{2}}{m} &= \binom{\binom{n}{2}}{(1-\epsilon) \binom{n}{2} p_n} \geq \left( \frac{1}{(1-\epsilon)p_n}\right)^{(1-\epsilon) \binom{n}{2} p_n}. \label{eq:size}
\end{align}

Using \eqref{eq:S-size}, \eqref{eq:size}, and the bound $n! \leq  \exp(n \log(n))$,
we have 
\begin{align*}
&\max_{m \in S_E} \frac{n!|S|}{\binom{\binom{n}{2}}{m}}\leq \frac{ \exp(n \log(n)) \left(\frac{1}{2} q_n^3 t_n n^3 \left(\frac{e}{t_n}\right)^{\frac{1}{2} q_n^2 t_n n^2} \right)^n}{ \left( \frac{1}{(1-\epsilon)p_n}\right)^{(1-\epsilon) \binom{n}{2} p_n} }\\
&= \exp \left\{ n \log(n) + n \left[\log\left( \frac{1}{2} q_n^3 t_n n^3 \right) + \frac{1}{2} q_n^2 t_n n^2 \log \left(\frac{e}{t_n}\right) \right] - (1-\epsilon) \binom{n}{2} p_n \log \left( \frac{1}{(1-\epsilon)p_n}\right) \right\}\\
&= \exp \left\{ n \log(n)  +  \Theta(n \log\left(n^{3+c-4\alpha} \right)) + \Theta (n^{3 + c - 3\alpha} \log(n)) - \Theta(n^{2-\alpha} \log(n)) \right\}\\
&= \exp \left\{ \Theta(n \log\left(n \right)) + \Theta (n^{3 + c - 3\alpha} \log(n)) - \Theta(n^{2-\alpha} \log(n)) \right\}.
\end{align*}
We need $3 + c - 3\alpha < 2-\alpha \iff c < 2\alpha -1$ for this bound to go to zero.

Summarizing, we require 
\[ \max\{0, 3\alpha-2\} < c < \min\{\alpha, 2\alpha -1\}.\] 
This is consistent for $\frac{1}{2} < \alpha < 1$. Then $\min\{\alpha, 2\alpha -1\} = 2\alpha - 1$. The proof is completed by choosing
\[c = \frac{1}{2} \left(\max\{0, 3\alpha-2\}  + 2\alpha -1\right). \qedhere\] 
\end{proof}

\section{Reconstruction from $2$-neighborhoods}\label{sec:2-neighbourhoods}
The following lemmas will be used to prove Theorem \ref{theorem:2-reconstructable} for the case $0 < \alpha < \frac{1}{2}$.
\begin{lemma}[\cite{Janson2000}]\label{lemma:diameter-2}
Let $p_n = c \frac{\sqrt{\log n}}{\sqrt{n}}$. If $c > \sqrt{2}$, then the diameter is $2$ with high probability. If $c < \sqrt{2}$, then the diameter is at least $3$ with high probability.
\end{lemma}

A canonical labeling of a graph involves assigning a unique label to each vertex, in a way that is invariant under isomorphism. 
\begin{lemma}[\cite{Czajka2008}]\label{lemma:canonical-labelling}
Suppose $p_n \leq \frac{1}{2}$ and $p_n = \omega \left(\frac{\log^4(n)}{n \log( \log (n))}\right)$, and let $G(V, E) \sim G(n, p_n)$. Then with high probability, any $u, v \in V$ have different degree neighborhoods. Furthermore, one can produce a canonical labeling of $G$ by sorting the vertices in lexicographic order by their degree neighborhoods.
\end{lemma}
\begin{remark}
The proof of Theorem \ref{theorem:3-neighbourhoods-original} (Theorem 4.5 in \cite{Mossel2019}) shows that the statement of Lemma \ref{lemma:canonical-labelling} is true for $p_n = \omega\left( \frac{\log^2(n)}{n}\right)$, which strengthens the result of \cite{Czajka2008}.
\end{remark}

\begin{proof}[Proof of Theorem \ref{theorem:2-reconstructable} for $0 < \alpha < \frac{1}{2}$]
By Lemma \ref{lemma:diameter-2}, the diameter is equal to $2$ with high probability, so we see the entire graph from any $2$-neighborhood. To label the graph, we produce the canonical labeling of any $2$-neighborhood, which is possible (and efficient) with high probability by Lemma \ref{lemma:canonical-labelling}. We then sort the vertices lexicographically by their degree neighborhoods in order to determine the canonical label of each vertex. This procedure exactly reconstructs $G$. 
\end{proof}

\begin{proof}[Proof of Theorem \ref{theorem:2-reconstructable} for $\frac{1}{2} < \alpha < \frac{3}{5}$]
The proof is similar to the proof of Theorem \ref{theorem:1-reconstructable}. This time, for $u,v \in V$ where $(u,v) \in E$, we define $L_{u,v}$ as the subgraph induced by those vertices which are at distance exactly $2$ from both $u$ and $v$. Lemma \ref{lemma:fingerprint} applies to the $L_{u,v}$ graphs as well. It therefore suffices to show that with high probability, whenever $(u,v)$ and $(x,y)$ are distinct edges, then $L_{u,v} \nsim L_{x,y}$.

Consider two distinct edges $(u,v)$ and $(x,y)$. Our goal is to show that $\mathbb{P}\left(L_{u,v} \sim L_{x,y} \right) = o \left(n^{-4} \right)$. Similarly to the proof of Theorem \ref{theorem:1-reconstructable}, we will identify two graphs $G_1$ and $G_2$ such that $\mathbb{P}\left(L_{u,v} \sim L_{x,y} \right) \leq \mathbb{P}\left(G_1 \subset G_2\right)$. Note that by revealing which vertices are at distance exactly $2$ from $u$ and $v$, we have not revealed any of the edges among those distance-$2$ neighbors.

Let $W_{ab}$ be the number of shared distance-$2$ neighbors of $a$ and $b$. If $L_{u,v} \sim L_{x,y}$, then $W_{uv} = W_{xy}$. 
Let  $d(a,b)$ be the graph distance between vertices $a$ and $b$, and let 
\[Z = \mathbb{1}\left\{d(u,x) = d(v,x) = 2\right\} + \mathbb{1}\left\{d(u,y) = d(v,y) = 2\right\}.\]
Suppose $W_{uv} = W_{xy} = \lambda + Z$, so that $\lambda$ counts all shared distance-$2$ neighbors of $u$ and $v$ except for $x$ and $y$. Let $Y = \left | \mathcal{N}_2(u) \cap \mathcal{N}_2(v) \cap \mathcal{N}_2(x) \cap \mathcal{N}_2(y)\right|$ be the number of shared distance-$2$ neighbors of $u$, $v$, $x$, and $y$. Suppose also that $Y = \mu$. Let $G_1$ be the subgraph induced by the $\lambda - \mu$ shared distance-$2$ neighbors of $u$ and $v$ that are not $x$, $y$, or a shared distance-$2$ neighbor of $x$ and $y$. Let $G_2$ be the subgraph induced by all distance-$2$ neighbors of $x$ and $y$. Finally, suppose $|\mathcal{N}_1(u)| + |\mathcal{N}_1(v)|  = d$. Similarly to the proof of Theorem \ref{theorem:1-reconstructable}, we claim
\begin{align}
&\mathbb{P}\left(L_{u,v} \sim L_{x,y} ~\big|~ W_{uv} = W_{xy} = \lambda + Z,  Y = \mu, E(G_1) = k , |\mathcal{N}_1(u)| + |\mathcal{N}_1(v)| = d\right)\nonumber \\
&\leq \binom{\lambda + 2}{\lambda - \mu} (\lambda - \mu)! p_n^{k-d} \leq \exp\left(\log(n) \left(\lambda - \alpha (k-d) \right) \right). \label{eq:bound-reconstructable-fingerprint}
\end{align}
To see the claim, consider the following way of determining which vertices belong in $G_1$ and $G_2$. First, determine which vertices are at distance-$1$ and distance-$2$ from $x$, and do the same for $y$. From this information, we determine which vertices are in $G_2$. We have not revealed any of the edges in $G_2$, unless $u$ and $v$ are both shared neighbors of $x$ and $y$. In that case, we have revealed the edge $(u,v)$. Next, we reveal the distance-$1$ neighbors of $u$ and $v$. This step reveals at most $|\mathcal{N}_1(u)| + |\mathcal{N}_1(v)| $ edges in $G_2$. Finally, to determine which vertices are in $G_1$, we consider all vertices that are not in $G_2$. From among those vertices, we place those that are shared distance-$2$ neighbors of $u$ and $v$ into $G_1$. This last step does not reveal any edges in $G_2$. Therefore, $G_2$ has at most $|\mathcal{N}_1(u)| + |\mathcal{N}_1(v)|$ edges revealed in this process, so the probability that any particular collection of $k$ edges exists is at most $p_n^{k-(|\mathcal{N}_1(u)| + |\mathcal{N}_1(v)| )}$.

We will now find high probability bounds on $\lambda$, $E(G_1)$, and $|\mathcal{N}_1(u)| + |\mathcal{N}_1(v)|$. Since $W_{uv} - 2 \leq \lambda \leq W_{uv}$, we will bound $W_{uv}$. Note that conditioned on $|\mathcal{N}_1(u)|$, $|\mathcal{N}_1(v)|$, and $|\mathcal{N}_1(u) \cap \mathcal{N}_1(v)|$, the number of shared distance-$2$ neighbors is distributed as a binomial random variable with parameters $\left(n - |\mathcal{N}_1(u) \cup \mathcal{N}_1(v)|, q_n \right),$ where
\begin{align*}
q_n &= 1 - (1-p_n)^{|\mathcal{N}_1(u) \cap \mathcal{N}_1(v)|}  \\
&~~~+ (1-p_n)^{|\mathcal{N}_1(u) \cap \mathcal{N}_1(v)|}\left(1 - (1-p_n)^{|\mathcal{N}_1(u) \setminus \mathcal{N}_1(v)| -1} \right) \left(1 - (1-p_n)^{|\mathcal{N}_1(v) \setminus \mathcal{N}_1(u)|-1} \right).
\end{align*}
To see this, observe that a candidate shared distance-$2$ neighbor must either be connected to some vertex in $\mathcal{N}_1(u) \cap \mathcal{N}_1(v)$, or failing this, must be connected to some vertex in $\mathcal{N}_1(u) \setminus \{\mathcal{N}_1(v)\cup \{v\}\}$ and some vertex in $\mathcal{N}_1(v) \setminus \{\mathcal{N}_1(u) \cup \{u\}\} $. To derive bounds on $W_{uv}$, we use bounds on $q_n$:
\begin{align*}
q_n &\leq 1 - (1-p_n)^{|\mathcal{N}_1(u) \cap \mathcal{N}_1(v)|}  + \left(1 - (1-p_n)^{|\mathcal{N}_1(u)|} \right) \left(1 - (1-p_n)^{|\mathcal{N}_1(v)|} \right)\\
q_n &\geq \left(1 - (1-p_n)^{|\mathcal{N}_1(u) \setminus \mathcal{N}_1(v)|-1} \right) \left(1 - (1-p_n)^{|\mathcal{N}_1(v) \setminus \mathcal{N}_1(u)|-1} \right).
\end{align*}
Fix $0 < \epsilon < 1$. We have that $|\mathcal{N}_1(u)|, |\mathcal{N}_1(v)| \leq (1+\epsilon) n^{1-\alpha}$ with probability $1-o(n^{-4})$. Noting that $\mathbb{E}\left[\left|\mathcal{N}_1(u) \cap \mathcal{N}_1(v) \right| \right] = (n-2)p_n^2 = \Theta\left(n^{1-2\alpha} \right)$ and applying Corollary \ref{corollary:chernoff}, we have  $|\mathcal{N}_1(u) \cap \mathcal{N}_1(v)| \leq n^{1-2\alpha + c}$ with probability $1-o(n^{-4})$ for $c$ satisfying $1-2\alpha + c > 0 \iff c > 2\alpha - 1$. Therefore, using the upper bound on $q_n$, we have that with probability $1-o(n^{-4})$,
\begin{align*}
W_{uv} &\leq (1+\epsilon) n  \left[1 - (1-p_n) ^{n^{1-2\alpha + c}}+ \left(1 - (1-p_n)^{(1+\epsilon)n^{1-\alpha}}\right)^2 \right] \\
&\leq (1+\epsilon)n  \left[1 - \exp\left(-2n^{1-3\alpha + c} \right) + \left(1 - \exp\left(-2(1+\epsilon)n^{1-2\alpha} \right)\right)^2\right] \\
&\leq (1+\epsilon)n  \left[2n^{1-3\alpha + c} + \left(2(1+\epsilon)n^{1-2\alpha} \right)^2\right]\\
&=\Theta\left(n^{2 -3\alpha + c} \right) + \Theta\left(n^{3-4\alpha} \right),
\end{align*}
where we have used the inequalities $1-x \geq e^{-2x}$ for $x \leq \frac{1}{2}$, and $1-e^{-x} \leq x$. Taking $c > 2\alpha - 1$ sufficiently small, the second term dominates and we arrive at $W_{uv} = O\left(n^{3-4\alpha}  \right)$ with probability $1-o(n^{-4})$.

Next, observe that $E(G_1)$ is a binomial random variable with parameters $\left( \binom{\lambda - \mu}{2}, p_n\right)$, conditioned on $W_{uv} - Y -Z = \lambda - \mu$. We similarly lower bound $W_{uv}$, using the lower bound on $q_n$. Observe that $|\mathcal{N}_1(u) \setminus N_1(v)|, |\mathcal{N}_1(v) \setminus N_1(u)| \geq (1-\epsilon) n^{1-\alpha}$ with probability $1-o(n^{-4})$. Therefore, with probability $1-o(n^{-4})$,
\begin{align*}
W_{uv} &\geq (1-\epsilon) n  \left(1 - (1-p_n)^{(1-\epsilon)n^{1-\alpha}-1}\right)^2 = \Omega\left(n^{3-4\alpha}  \right).
\end{align*}
By similar reasoning, we have $\mu = O\left(n^{5 - 8\alpha} \right)$ with probability $1-o(n^{-4})$ if the four vertices are distinct, and otherwise $\mu = O\left(n^{4 - 6\alpha} \right)$ with probability $1-o(n^{-4})$. Since $5 - 8 \alpha , 4 - 6\alpha< 3 - 4 \alpha$, we have $\binom{\lambda - \mu}{2} = \Theta\left(n^{6-8\alpha} \right)$, so that $E(G_1) = \Theta\left(n^{6 - 9 \alpha} \right)$, with probability $1-o(n^{-4})$.

Finally, examining the bound \eqref{eq:bound-reconstructable-fingerprint}, we first compare $E(G_1)$ to $|\mathcal{N}_1(u)| + |\mathcal{N}_1(v)| $. We see that the order of $E(G_1)$ dominates when
\[6 - 9 \alpha > 1 -\alpha \iff \alpha < \frac{5}{8}. \]
Since $\frac{5}{8} > \frac{3}{5}$, the edge count indeed dominates. Next, we compare $\lambda$ to $E(G_1)$. We see that the order of $E(G_1)$ dominates when 
\[6 - 9 \alpha > 3 - 4 \alpha \iff \alpha < \frac{3}{5}.\]
We have shown that for a given pair of distinct edges $(u,v)$ and $(x,y)$, we have $\mathbb{P}(L_{u,v} \sim L_{x,y}) = o(n^{-4})$. Taking a union bound completes the proof.
\end{proof}

\begin{proof}[Proof of Theorem \ref{theorem:2-non-reconstructable}]
The proof is similar to the proof of Theorem \ref{theorem:1-non-reconstructable-part-1}. Again we set $S_E = \{m \in \mathbb{N} : \left|m - \binom{n}{2} p_n \right| < \epsilon \binom{n}{2}\}$. It remains to find a set $S$ of collections of $2$-neighborhoods such that $\mathcal{N}_2(G) \in S$ with high probability, and show that
\[\max_{m \in S} \frac{n!|S|}{\binom{\binom{n}{2}}{m} } = o(1).\]

We now construct the set $S$. Fix $0 < \epsilon < 1$ to be determined, and let $q_n = (1+\epsilon) p_n$. For some $c > 0$ to be determined, let $t_n = (1 + n^c) p_n$.  Let $X(v)$ and $Y(v)$ respectively be the number of vertices and edges in the graph $\mathcal{N}_2(v)$. Let $S$ be the set of possible $2$-neighborhood collections where for each $v \in V$, the following three conditions hold: (1) $\deg(v) < q_n n$; (2) $\deg(w) < q_n n$ for each $w \sim v$; (3) $Y(v) < \frac{1}{2} q_n^4 t_n n^4 + X(v) -1$.
We need to show that the collection of $2$-neighborhoods of $G$ appears in $S$ with high probability. As was shown in the proof of Theorem \ref{theorem:1-non-reconstructable-part-1}
\begin{align*}
\mathbb{P}\left(\cap_{v \in V} \left\{ \deg(v) <  q_n n  \right\}\right) &= 1 -o(1).\end{align*}
Therefore, the first two properties hold for every neighborhood with high probability. Let $E_v$ be the event that properties (1) and (2) hold in the $2$-neighborhood around $v$. 

Next, we bound the number of edges. Our goal is to lower-bound
\[\mathbb{P} \left(Y(v) \geq \frac{1}{2} q_n^4 t_n n^4 + X(v) - 1~\Big|~ E_v \right).\] 
We reveal the neighborhood $\mathcal{N}_2(v)$ in a particular way. We start from $v$, and reveal its neighbors (and therefore the edges $\{(v, w) : w \sim v\}$). Then we choose an arbitrary order for the neighbors. Starting from the first neighbor $w_1$, we reveal all of its additional neighbors apart from $v$. Next, we come to the second neighbor $w_2$, and reveal all of its neighbors apart from $v$ and $w_1$. Continuing through the neighbors of $v$, we reveal all the vertices of the neighborhood using $X(v) -1$ edges (since each vertex except $v$ is revealed by exactly one other vertex). The number of additional edges in $\mathcal{N}_2(v)$ is dominated by a binomial random variable with parameters $\left( \binom{\left \lfloor (q_nn)^2\right \rfloor}{2}, p_n \right)$. Let $Z \sim \text{Bin} \left( \binom{\left \lfloor (q_nn)^2\right \rfloor}{2}, p_n \right)$. We therefore have
\begin{align*}
\mathbb{P}\left( Y(v) \geq  \frac{1}{2} q_n^4 t_n n^4 + X(v) -1 ~\big|~ E_v \right) &\leq \mathbb{P} \left(Z \geq  \frac{1}{2} q_n^4 t_n n^4\right)\\
&\leq \mathbb{P} \left(Z \geq  \binom{\left \lfloor (q_n n)^2 \right \rfloor}{2} t_n \right)\\
&\leq \exp \left(- \frac{n^{2c}}{2 + n^c} \binom{\left \lfloor (q_n n)^2 \right \rfloor}{2} p_n  \right)\\
&= \exp \left(- \Theta\left(n^{4 + c-5\alpha } \right) \right).
\end{align*}
For this to be a high probability bound, we need $4+ c -5\alpha > 0 \iff c > 5\alpha -4$. We then have
\begin{align*}
\mathbb{P}\left(\bigcap_{v \in V} \left\{ \deg(v) <  q_n n, Y(v) < \frac{1}{2} q_n^4 t_n n^4 + X(v) -1  \right\}\right) &= 1 -o(1),
\end{align*}
We conclude that $\mathcal{N}_2(G) \in S$ with high probability.

We now bound $|S|$. Consider one collection of $2$-neighborhoods in the set $S$, and a particular vertex $v \in V$. There are $q_n n$ choices for the number of neighbors of $v$. Next, there are fewer than $q_n n (q_n n -1) < (q_n n)^2$ vertices which are at distance $2$ from $v$. We reveal the distance-$2$ neighbors sequentially, as described above. Each distance-$2$ neighbor is assigned to one of the distance-$1$ neighbors, according to which distance-$1$ neighbor revealed it.
Using a stars-and-bars argument, the number of ways this assignment can happen is upper-bounded by
\[\binom{q_n n (q_n n -1) + q_n n -1}{q_n n -1} = \binom{(q_n n)^2 -1}{q_n n -1} <  \binom{(q_n n)^2}{q_n n}.\]
So far, $X(v) -1$ edges have been revealed. There are fewer than $q_n^4 t_n n^4$ remaining edges, and at most $\binom{\binom{(q_n n)^2 + 1}{2}}{q_n^4 t_n n^4}$ choices for their locations, as long as $q_n^4 t_n n^4 < \frac{1}{2}\binom{(q_n n)^2 + 1}{2}$, which is satisfied for $c < \alpha$. We therefore have
\begin{align*}
&|S| \leq \left[q_n n \cdot (q_n n)^2 \cdot \binom{(q_n n)^2}{q_n n}  \cdot q_n^4 t_n n^4 \cdot \binom{\binom{(q_n n)^2 + 1}{2}}{q_n^4 t_n n^4} \right]^n\\
&\leq \left[q_n^7 t_n n^7  \left(e q_n n \right)^{q_n n}  \left(\frac{e\binom{(q_n n)^2 + 1}{2}}{q_n^4 t_n n^4} \right)^{q_n^4 t_n n^4} \right]^n\\
&= \exp \left\{n \left(\log \left(q_n^7 t_n n^7 \right) + q_n n \log\left(e q_n n \right) + q_n^4 t_n n^4 \log \left(\frac{e\binom{(q_n n)^2 + 1}{2}}{q_n^4 t_n n^4} \right) \right) \right\}\\
&= \exp \left\{n \log \left(q_n^7 t_n n^7 \right) + (1+\epsilon)n^{2-\alpha}\log(e(1+\epsilon)) + (1-\alpha) (1+\epsilon)n^{2-\alpha}\log(n)  + q_n^4 t_n n^5 \log \left(\frac{e\binom{(q_n n)^2 + 1}{2}}{q_n^4 t_n n^4} \right)  \right\}\\
&= \exp \left\{(7+c -8\alpha)\Theta \left(n \log(n) \right)+ \Theta\left(n^{2-\alpha} \right) + (1-\alpha) (1+\epsilon)n^{2-\alpha}\log(n) + \Theta\left(n^{5+c-5\alpha}\log(n) \right)  \right\}\\
&= \exp \left\{ \Theta\left(n^{2-\alpha} \right) + (1-\alpha) (1+\epsilon)n^{2-\alpha}\log(n) + \Theta\left(n^{5+c-5\alpha}\log(n) \right)  \right\}.
\end{align*}
As noted in the proof of Theorem \ref{theorem:1-non-reconstructable-part-1},
\begin{align*}
\min_{m \in S_E} \binom{\binom{n}{2}}{m} &\geq \left( \frac{1}{(1-\epsilon)p_n}\right)^{(1-\epsilon) \binom{n}{2} p_n}  \geq \exp \left(\frac{1}{2}(1-\epsilon) (n-1)^2 n^{-\alpha} \log\left( \frac{n^{\alpha}}{(1-\epsilon)} \right)\right).
\end{align*}
Using this bound and the bound on $|S|$, we have
\begin{align*}
&\max_{m \in S} \frac{n!|S|}{\binom{\binom{n}{2}}{m} } 
\leq \exp \Big\{ \Theta\left(n^{2-\alpha} \right) + (1-\alpha) (1+\epsilon)n^{2-\alpha}\log(n) + \Theta\left(n^{5+c-5\alpha}\log(n) \right) \\
&~~~~~~~~~~~~~~~~~~~~~~~~~~~~~ - \frac{1}{2}(1-\epsilon) (n-1)^2 n^{-\alpha} \log\left( \frac{n^{\alpha}}{(1-\epsilon)} \right) \Big\}\\
&= \exp \left\{ \Theta\left(n^{2-\alpha} \right) + (1-\alpha) (1+\epsilon)n^{2-\alpha}\log(n) + \Theta\left(n^{5+c-5\alpha}\log(n) \right) - \frac{\alpha}{2}(1-\epsilon) (n-1)^2 n^{-\alpha} \log\left( n \right) \right\}.
\end{align*}
Therefore, it suffices to have $5 + c - 5\alpha < 2 -\alpha \iff c < 4\alpha -3$ as well as 
\[(1-\alpha)(1+\epsilon) < \frac{\alpha}{2} (1-\epsilon) \iff \alpha > \frac{2 + 2 \epsilon}{3 + \epsilon}.\]
Taking $\epsilon = \frac{1}{6}$ satisfies the above. We then obtain
\begin{align*}
\max_{m \in S} \frac{n!|S|}{\binom{\binom{n}{2}}{m} }  &\leq \exp \left(\Theta\left(n^{2-\alpha}\right) - \Theta\left(n^{2-\alpha} \log(n) \right)\right) = o(1).
\end{align*}

Summarizing, we require $\max\{0, 5\alpha-4\} < c < \min\{\alpha, 4\alpha -3\}$.
This is a consistent condition for $\frac{3}{4} < \alpha < 1$. We then choose $c = \frac{1}{2} \left(\max\{0, 5\alpha-4\} +  \min\{\alpha, 4\alpha -3\}\right)$. \qedhere
\end{proof}

\begin{proof}[Proof of Proposition \ref{proposition:overlap}]
Let $\beta = \frac{1}{2}\left(\alpha + \frac{2}{3} \right)$, so that $\frac{2}{3} < \beta < \alpha$. We will show that with probability $1-o(1)$, there exists some $k_n \leq n^{1-\beta}$ such that there are at least $\frac{1}{2}n^{\beta}$ stars of degree $k_n$ in $G_n$, implying that the $1$-neigborhoods are not unique.

Consider an arbitrary vertex $v$. By the Markov inequality,
\begin{align}
\mathbb{P}\left(\deg(v) > n^{1-\beta} \right) \leq \frac{n^{1-\alpha}}{n^{1-\beta}} = n^{\beta - \alpha} = o(1). \label{eq:degree-Markov}
\end{align}
Let $S_v$ be the indicator that $\mathcal{N}_1(v)$ is a star of degree at most $n^{1-\beta}$. Let $Z \sim \text{Bin}\left(\binom{n^{1-\beta}}{2}, p_n \right)$. Observe that $\mathbb{E}[Z] < \frac{1}{2}n^{2-2\beta - \alpha} = o(1)$. Applying \eqref{eq:degree-Markov}, we have
\begin{align*}
\mathbb{P}\left(S_v = 1\right) &= \mathbb{P}\left(S_v = 1 ~|~ \deg(v) \leq  n^{1-\beta} \right) \left(1 - o(1)\right)+ o(1)\\
&\geq \mathbb{P}\left(S_v = 1 ~|~ \deg(v) =  n^{1-\beta} \right)\left(1 - o(1)\right) \\
&= 1 - \mathbb{P}\left(Z \geq 1 \right) - o(1)\\
&\geq 1 - \mathbb{E}[Z] - o(1) = 1 - o(1).
\end{align*}
By the Markov inequality, $\mathbb{P}\left(\sum_{v \in V} S_v < \frac{3n}{4} \right) = \mathbb{P}\left(\sum_{v \in V} (1- S_v) > \frac{n}{4} \right) = o(1).$
We conclude that with probability $1-o(1)$, there are at least $\frac{3n}{4}$ vertices $v$ such that $\mathcal{N}_1(v)$ is a star of degree at most $n^{1-\beta}$. Then by  the Pigeonhole Principle, there is at least one $k_n \leq n^{1-\beta}$ such that the number of stars of degree $k_n$ is at least $\frac{3n}{4(n^{1-\beta} + 1)} \geq \frac{n}{2n^{1-\beta}} $.
\end{proof}

\section{Finding the Neighborhood Centers}\label{section:find-center}
Previously we assumed that the central vertex $v$ of each neighborhood $\mathcal{N}_r(v)$ was labeled. In this section, we show how to determine the center if it is not labeled.
\subsection{Finding centers of $1$-neighborhoods}
The following lemma shows that in order to identify the center of each neighborhood, we should simply take the vertex that is connected to all others.
\begin{lemma}\label{lemma:1-neighborhood-find-center}
Suppose $p_n = n^{-\alpha}$ for $0 < \alpha < 1$. Then with high probability, for all $v \in V$, the vertex $v$ is the only other vertex connected to every other vertex in $\mathcal{N}_1(v)$.
\end{lemma}
\begin{proof}
Consider $v \in V$ and the neighborhood $\mathcal{N}_1(v)$. The degree of $v$ is distributed according to a binomial distribution with parameters $(n-1,p_n)$. Furthermore, for $0 < \epsilon < 1$, it holds that $\deg(v) \geq (1-\epsilon)(n-1)p_n$ with high probability.
Conditioned on $\deg(v) = m \geq (1 - \epsilon) p_n (n-1)$, the probability that a given neighbor $w \sim v$ is connected to all other neighbors of $v$ is then $p_n^{m-1}$. The claim follows by a union bound.
\end{proof}

\subsection{Finding centers of $2$-neighborhoods}
To identify the center of a neighborhood, we first prune the neighborhood by removing all vertices with degree less than $\frac{1}{2}n^{1-\alpha}$. We then return the vertex with the highest degree in the subgraph induced by the remaining vertices (which is also the only vertex connected to all the others). Recall that if $\alpha < \frac{1}{2}$, then the graph $G$ has diameter $2$ with high probability.
\begin{lemma}\label{lemma:2-neighborhood-find-center}
Suppose $p_n = n^{-\alpha}$ for $\frac{1}{2} < \alpha < 1$. Then with high probability, for all $v \in V$, the above approach recovers the center $v$ in $\mathcal{N}_2(v)$.\end{lemma}
\begin{proof}
Consider the neighborhood $\mathcal{N}_2(v)$. Observe that with high probability, the total number of vertices in the $2$-neighborhood is of order at most $(np_n)^2 = n^{2 - 2\alpha}$. Each neighbor $u$ of $v$ has $(n-1)p_n = \Theta\left(n^{1-\alpha}\right)$ expected neighbors, with exponential concentration. On the other hand, each vertex $w$ which is at a distance of $2$ from $v$ has $1 + \left(| V(\mathcal{N}_2(v))| -3\right)p_n = O\left(n^{2-3\alpha} \right)$ expected neighbors in $\mathcal{N}_2(v)$. Comparing exponents, we have $2 - 3 \alpha < 1 - \alpha$ for $\alpha >\frac{1}{2}$. Therefore, by removing all vertices with fewer than $\frac{1}{2} n^{1-\alpha}$ neighbors, we remove all vertices which are at distance $2$ from $v$, while not removing any neighbors of $v$, with high probability. By a union bound, this process leaves us with $\mathcal{N}_1(v)$ for each $v \in V$ with high probability. By Lemma \ref{lemma:1-neighborhood-find-center}, we can identify the center of each pruned neighborhood.
\end{proof}

\section{Conclusion and Open Problems}
In this paper, we have studied the problem of shotgun assembly of \ER graphs from $1$- and $2$-neighborhoods. We have established regimes for reconstructability and non-reconstructability in terms of the edge probability $p_n = n^{-\alpha}$. Our work leaves some open problems: 
\begin{enumerate}
\item Is reconstruction from $2$-neighborhoods possible for $\alpha = \frac{1}{2}$ or $\frac{3}{5} \leq \alpha \leq \frac{3}{4}$? 
\item Our results show that there is an efficient algorithm for reconstruction from $2$-neighborhoods in the regime $0 < \alpha < \frac{1}{2}$. Can we find efficient algorithms for reconstruction in other cases? 
\end{enumerate}

\textbf{Acknowledgements}: We thank the anonymous reviewer for a very careful review. The reviewer's comments caught several errors, and improved the clarity and presentation of this work.

\end{document}